\documentclass{article}
\usepackage{arxiv}

\usepackage[utf8]{inputenc} 
\usepackage[T1]{fontenc}    
\usepackage{hyperref}       
\usepackage{url}            
\usepackage{booktabs}       
\usepackage{amsfonts}       
\usepackage{nicefrac}       
\usepackage[english]{babel}
\usepackage{ae}
\usepackage{eucal}
\usepackage[dvips]{graphicx}
\usepackage{epsfig}
\usepackage{graphicx}
\usepackage{amsfonts,amsthm}
\usepackage{amssymb}
\usepackage{a4}
\usepackage{apu}
\usepackage{xcolor}
\usepackage{amsmath}
\usepackage{mathtools}
\usepackage{multicol}

\usepackage{dcolumn,amsthm}

\renewenvironment{proof}[1][Proof]{\noindent\textit{#1. } }{\hfill$\square$}

\newtheoremstyle{theorem}{6pt}{6pt}{\rm}{}{\sffamily}{ }{ }{}
\theoremstyle{theorem}
\newtheorem{theorem}{\sc Theorem}[section]

\newtheoremstyle{lemma}{6pt}{6pt}{\rm}{}{\sffamily}{ }{ }{}
\theoremstyle{lemma}
\newtheorem{lemma}{\sc Lemma}[section]

\newtheoremstyle{example}{6pt}{6pt}{\rm}{}{\sffamily}{ }{ }{}
\theoremstyle{example}

\newtheoremstyle{corollary}{6pt}{6pt}{\rm}{}{\sffamily}{ }{ }{}
\theoremstyle{corollary}

\newtheoremstyle{definition}{6pt}{6pt}{\rm}{}{\sffamily}{ }{ }{}
\theoremstyle{definition}

\newtheoremstyle{remark}{6pt}{6pt}{\rm}{}{\sffamily}{ }{ }{}
\theoremstyle{remark}

\newtheoremstyle{approximation}{6pt}{6pt}{\rm}{}{\sffamily}{ }{ }{}
\theoremstyle{approximation}

\newtheoremstyle{scheme}{6pt}{6pt}{\rm}{}{\sffamily}{ }{ }{}
\theoremstyle{scheme}

\usepackage{csquotes}

\title{On the explicit solutions of separation of variables type for the incompressible 2D Euler equations}

\author{
   Tomi Saleva \\
   Department of Physics and Mathematics\\
   P.O. box 111, FI-80101 Joensuu,Finland  \\
   \text{tomisal@student.uef.fi} \\
      \And
   Jukka Tuomela \\
   Department of Physics and Mathematics\\
   P.O. box 111, FI-80101 Joensuu,Finland  \\
   \text{jukka.tuomela@uef.fi} \\
}

\begin{document}

\maketitle

\begin{abstract}
We study explicit solutions to the 2 dimensional Euler equations in the Lagrangian framework. All known solutions have been of the separation of variables type, where time and space dependence are treated separately. The first such solutions were known already in the 19th century. We show that all the solutions known previously belong to two families of solutions and introduce three new families of solutions. It seems likely that these are all the solutions that are of the separation of variables type.
\end{abstract}
\thanks{The first author was supported by the North Karelia Regional Fund of Finnish Cultural Foundation}

\section{Introduction}
\label{sec;introduction}

We will continue our analysis of explicit solutions of the incompressible Euler equations which was started in \cite{maju}. For a general overview of various aspects of Euler equations from the mathematical point of view we refer to the survey \cite{constantin}. There are two ways to think about the Euler (and Navier--Stokes) equations: either one focuses on the velocity field or the fluid particles. The first approach is the Eulerian description and the second the Lagrangian description of the fluid flow. Below we will concentrate on the Lagrangian framework; for a physical treatment of this topic we refer to \cite{bennett}.

Since for nonlinear PDE it is typically very difficult to find any explicit solutions, it has been found convenient to relax the conditions in the Lagrangian description somewhat. So, instead of requiring that the determinant of the differential of the map from the Lagrangian to the Eulerian coordinates is one, we only demand that it is independent of time. Let us call this approach the quasi Lagrangian description of the flow. The goal was then to use this extra freedom to find more explicit solutions. Note that this quasi Lagrangian description still has the full information about the flow. Simply the coordinates that are used to describe the flow have no intrinsic physical meaning: they are just arbitrary, but convenient, coordinates.

The first explicit solutions of this type were already found in 19th century by Gerstner and Kirchhoff \cite{G,K}. Actually Kirchhoff's solution is so simple that one can also explicitly compute the Eulerian description of the flow, but Gerstner's solution is genuinely a quasi Lagrangian solution.
These  solutions were then used to analyze more complicated situations with perturbation techniques. Also, Gerstner's solution has the remarkable property that it can be used to model the interface between two different fluids, like air and water.

Apparently no really new explicit solutions were found before the paper by Abrashkin and Yakubovich  in 1984 \cite{AY}. These solutions were a  generalization of both Kirchhoff's and Gerstner's solution. After this these types of solutions have been analyzed and generalized using harmonic maps, see for example  \cite{Abrashkin,AC,CM} and the references therein. Also group theory has been used in the analysis of solutions \cite{andreev}. The role of analytic functions has been quite strong in these constructions, which is in some sense natural since already in the 19th century it was noticed that analytic functions could be used to analyze certain two dimensional flow problems. 

All explicit quasi Lagrangian solutions that were constructed turned out to be of the separation of variables type: the time dependence and spatial dependence could be treated separately. However,  while for the spatial part one could find solutions using complex functions, there was no natural role for complex functions for the time dependent part. Also in \cite{maju} it was shown that also in the spatial domain the complex functions were not as essential as was previously thought.

Since complex functions were used in the description of the solutions, it was natural to also consider harmonic functions. In \cite{maju} we showed that if the map in the plane is both area preserving and harmonic, then it is necessarily affine. So a harmonic Lagrangian solution is like Kirchhoff's solution. On the other hand, Gerstner's solution is also harmonic, so that indeed by relaxing the conditions one obtains essentially new solutions with the quasi Lagrangian framework. However, harmonic functions are not really essential in the description of quasi Lagrangian solutions, as we will see below.

In the present article we do not use complex analysis at all. The reason is simple: complex analysis is not needed, and the analysis given below is quite naturally formulated in terms of real functions and real variables. The new families of quasi Lagrangian solutions given below come naturally from the systematic analysis of the problem in the real domain. Indeed, the only reason we can think of why these families were not discovered previously is that their description using complex functions would be quite awkward. Also the harmonicity of functions plays no role in these new solutions and finally our analysis is local so the question if given maps are analytic or merely differentiable is irrelevant in the present context. 

Since we are not using complex functions, it is not so easy to compare our solutions to the previously known cases. For example, 
if the reader takes a look at our formula \eqref{2x2-sol} and compares it to the essentially equivalent formula (25) in \cite[Theorem 3]{CM}, then it is clear that the equivalence is not immediately obvious. Anyway it seems that all the previously known solutions reduce either to the situation described in Section \ref{k=2} (this could be called the Kirchhoff type case) or to the family of solutions given in Theorem \ref{ell-ratk} (the Gerstner type case). Solutions of these types can be found using harmonic maps, and even though there have previously been hints that even more complicated solutions exist \cite{AC,maju}, we show in this paper how they can all be reduced to these cases. Also, as far as we know, the families of solutions in Theorems \ref{2x3-th}, \ref{hyper-aika} and \ref{para-ratk} are new, the first of which is a generalization of the Kirchhoff type. Thus we have four essentially different families of solutions, and apparently they give all the quasi Lagrangian solutions that are of the separation of variables type. We will not prove that there cannot be more solutions of this type but we discuss below why we think that the existence of essentially different solutions is unlikely. 

One could also ask how big the families of solutions are. One way to measure this is to count the number of the arbitrary functions and constants in the general solution. Another physically interesting point of view is to ask if one can find a solution with prescribed vorticity. For all four families of solutions, we can compute a certain PDE, such that if the vorticity is a solution to this PDE, then there is an explicit solution with this vorticity. In one case the relevant PDE is obvious while in the remaining three cases we have used the algorithm \textsf{rifsimp} \cite{rif0}, which is based on the ideas of the formal theory of PDE \cite{seiler}.

The Lagrangian framework has been and is still being used in many different contexts. In addition to the bulk flow, an interesting aspect is to model the flow in the presence of an air/water interface. In some other  applied problems the equations are not precisely the Euler equations; for example,  in the large scale ocean current and meteorological  problems it is important to take into account the Coriolis effect. Anyway we hope that our new solutions will be useful also in these more general problems.  For various aspects of the applications of the Lagrangian point of view we refer to \cite{ab,AO1,AO2,C,CMo,henry,kluczek} and the many references therein. 

The paper is organized as follows. In Section 2 we collect some necessary background material. In Section 3 we formulate the problem precisely and analyze the first family of solutions, of which the Kirchhoff type case is a special case. Then in Sections 4 and 5 we show that there are three more families of solutions, one of which is the Gerstner family and the other two are new. Finally, in Section 6 we discuss to what extent one can prescribe the vorticity of the solutions.

\section{Preliminaries}

\subsection{Notation}
  Let $v=(v^1,\dots,v^m)\,:\,\mathbb{R}^n\to \mathbb{R}^m$ be some map and $\alpha\in\mathbb{N}^n$ a multiindex.  For spatial derivatives we use the jet notation:
\[
   v^k_\alpha=\frac{\partial^{|\alpha|} v^k}{\partial^{\alpha_1} z_1\dots\partial^{\alpha_n} z_n}\, .
\]
If $v$ depends also on time we may use for the time derivative $v_t$ or $v'$, whichever is more convenient in a given formula. For functions $a$ that depend only on time we always use $a'$ for derivative. 

In the analysis we will meet the Cauchy--Riemann equations in two different forms so to avoid confusion let us introduce the following terminology. Let $v\,:\,\mathbb{R}^2\to \mathbb{R}^2$ be some map and consider the following PDE
\[
  \begin{cases}
  v^1_{10}-v^2_{01}=0\\
  v^1_{01}+v^2_{10}=0
  \end{cases}\quad\mathrm{and}\quad
   \begin{cases}
  v^1_{10}+v^2_{01}=0\\
  v^1_{01}-v^2_{10}=0
  \end{cases}\, .
\]
The left system will be called the CR system and the right system the anti CR system. The solutions to the left system are CR maps and to the right system anti CR maps. 
Let us also introduce  the  rotations and reflections
\[
  M(\theta)= \begin{pmatrix}
  \cos(\theta)&-\sin(\theta)\\
  \sin(\theta)&\cos(\theta)
  \end{pmatrix}\quad\mathrm{and}\quad
  \hat M(\theta)= \begin{pmatrix}
  \cos(\theta)&\sin(\theta)\\
  \sin(\theta)&-\cos(\theta)
  \end{pmatrix}
\]
where $\theta$ is a function of time. 

The minors of various matrices appear frequently in the computations and so it is convenient to recall some facts about them. Let $A\in \mathbb{R}^{2\times k}$ and let us denote the columns of $A$ by $A_j$; then the minors of $A$ will be denoted by $p_{ij}=\det(A_i,A_j)$. Also when $v\,:\,\mathbb{R}^2\to\mathbb{R}^k$ is some map then the minors of its differential $dv$ are denoted by
\[
   g_{ij}=\det(\nabla v^i,\nabla v^j)
\]
In the analysis below we will repeatedly use the following simple facts.
\begin{lemma}
Suppose that $A_i\ne 0$ and $p_{ij}=p_{ik}=0$; then also $p_{jk}=0$.  In addition
\[
    p_{ij}p_{k\ell}-p_{ik}p_{j\ell}+p_{i\ell}p_{jk}=0.
\]
If $\varphi=Av$ then we have the Cauchy--Binet formula
\[
\det(d\fii)=\sum _{1\leq i<j\leq k} p_{ij}g_{ij}.
\]
\label{minori-riippuvuus}
\end{lemma}

\subsection{Overdetermined PDE}

In some computations below we have used the  algorithm \textsf{ rifsimp} \cite{rif0} which is implemented in {\sc Maple}.  The
acronym \textsf{rif} means \emph{reduced involutive form} and the word involutive refers to the fact that general systems of PDE can be transformed to an involutive form. 
 For a comprehensive overview of overdetermined or general PDE systems we refer to \cite{seiler}. 

An analogous situation arises in polynomial algebra \cite{colios}. A polynomial system generates an ideal, which in turn
defines the corresponding variety. Now computing the Gr\"obner basis of the ideal gives a lot of information
about the variety. Similarly the involutive form can reveal important information about the structure of the solution set.   Intuitively one may think about computing the involutive form of a system of PDE
like computing the Gr\"obner basis of an ideal.

\subsection{Euler equations}
Let us consider the incompressible Euler equations
\begin{equation}
    \begin{aligned}
    u_t+u\nabla u+\nabla p=0\\
    \nabla\cdot u=0
\end{aligned}
\label{euler}
\end{equation}
in some domain $\Omega\subset\mathbb{R}^n$. This is called the Eulerian description of the flow and the coordinates of $\Omega$, denoted $x$, are the Eulerian coordinates. Below we will consider another description which is almost the Lagrangian description of the flow. 

Let $D\subset\mathbb{R}^n$ be another domain and let us consider a family of diffeomorphisms $\varphi^t\,:\,D\to \Omega_t=\varphi^t(D)$. The coordinates in $D$ are denoted by  $z$.\footnote{The coordinates $z$ are sometimes called labels, and $D$ is then the labelling domain.} We can also define 
\[
  \varphi\,:\,D\times \mathbb{R}^n\to \mathbb{R}^n\quad,\quad
  \varphi(z,t)=\varphi^t(z)\, .
\]
Now given such $\varphi$ we can define the associated vector field $u$ by the formula
\begin{equation}
\frac{\partial}{\partial t} \fii(z,t)=u(\fii(z,t),t)\,.
\label{siirto}    
\end{equation}
Our goal is to find maps $\varphi$ such that $u$ solves the Euler equations in the two dimensional case. To state the relevant conditions, 
let us introduce the following matrices:
\[
   P_1=\begin{pmatrix}
         (\fii^1_{10})'&(\fii^1_{01})'\\
          \fii^1_{10}&\fii^1_{01}
          \end{pmatrix}\quad\mathrm{and}\quad
          P_2=\begin{pmatrix}
         (\fii^2_{10})'&(\fii^2_{01})'\\
          \fii^2_{10}&\fii^2_{01}
          \end{pmatrix}\, .
\]
Straightforward computations show (see for example \cite{maju} for details) that we get the following conditions.
\begin{theorem} Let $h=\det( P_1)+\det(P_2)$ and let us suppose that  the following conditions are satisfied:
\begin{align*}
    \partial_t h=0\, ,\\
    \partial_t  \det(d\varphi)=0\, .
\end{align*}
Then $u$ given by \eqref{siirto} is a solution to \eqref{euler}.
\label{kriteeri}
\end{theorem}
In this case the Lagrangian description of the flow is given by the map $\Phi^t=\varphi^t\circ (\varphi^0)^{-1}$. Note that without loss of generality we can suppose that $ \det(d\varphi)>0$. 
It is also interesting to formulate the above condition in terms of vorticity. Recall that in the $x$ coordinates the vorticity $\hat\zeta=u^2_{10}-u^1_{01}$. 
Let us denote by $\zeta$ the vorticity in the $z$ coordinates, i.e. $\zeta=\hat\zeta\circ\fii^t$. Recall that in 2 dimensions, if $u$ is a solution to the Euler equations, then
\[
     \hat\zeta_t+\langle u,\nabla \hat\zeta\rangle=0\, .
\]
In the $z$ coordinates this simply means that  $\zeta_t=0$. But then again straightforward computations show that in fact
\[
   \zeta=\frac{h}{\det(d\varphi)}\, .
\]
Hence the condition of the previous Theorem could also be formulated using the vorticity instead of $h$.

In what follows we will try to find the most general solution of the given form. Then it is important to remember that the domain $D$ is simply some parameter domain which has no physical significance. Hence one can look for ''simplest'' possible parameter domain. For future reference let us record this observation as 
\begin{lemma} Let  $\psi\,:\,\hat D\to D$ be an arbitrary diffeomorphism  and let $\tilde\varphi^t=\varphi^t\circ\psi$. Then $\tilde\varphi$  provides solutions to the Euler equations via formula \eqref{siirto} if and only if $\varphi$ does.
\label{koordi}
\end{lemma}
\begin{proof}
This is because $\det(d\tilde\varphi)=\det(d\varphi)\det(\psi)$ and $\tilde\zeta=\zeta$.
\end{proof}

\section{General formulation of the problem}

Let us consider the maps of the following  form
\begin{equation}
    \varphi(z,t)=A(t)v(z)\, ,
\label{yrite}
\end{equation}
where $A(t)\in\mathbb{R}^{2\times k}$, $v\,:\, D\to\mathbb{R}^k$ and $D\subset\mathbb{R}^2$ is some  coordinate domain. Since all the analysis is local, the precise nature of $D$ is not important in our context. We will try to find maps $\varphi$ such that the corresponding vector field $u$ defined by the formula \eqref{siirto} is a solution to the Euler equations. Hence we should find $A$ and $v$ such that the conditions in Theorem \ref{kriteeri} are satisfied.
 Since we want that $\det(d\varphi)\ne0$, this necessarily implies that $\mathsf{rank}(A)=\mathsf{rank}(dv)=2$.

The strategy we use to tackle this problem is described now. 
Since $\det(d\fii)$ is independent of time, Lemma \ref{minori-riippuvuus} implies that
\begin{equation}
\partial_t \det(d\fii)=\sum _{1\leq i<j\leq k} p_{ij}'g_{ij}=0.
\label{yleinen-det}
\end{equation}
Now, if we fix any $t$, we obtain from this formula a homogeneous linear equation for the minors of $dv$:
\begin{equation}
\sum _{1\leq i<j\leq k}\alpha _{ij}g_{ij}=0,\quad \forall z \in D.
\label{yleinen-paikka}
\end{equation}
We also recall from \cite{maju} that if $\varphi$ is given by \eqref{yrite} then 
\begin{equation}
\begin{aligned}
&\partial _t h = \sum _{1\leq i<j\leq k} q_{ij}g_{ij} = 0
\,\ \mathrm{where}\\ 
&q_{ij}=a_{1i}''a_{1j}-a_{1j}''a_{1i}+a_{2i}''a_{2j}-a_{2j}''a_{2i}\ .
\end{aligned}
\label{yleinen-dh}
\end{equation}
This condition also gives equations of form \eqref{yleinen-paikka} when $t$ is fixed. We conclude that for the most general solution we should look for solutions for which $v$ satisfies a system of constraints of form \eqref{yleinen-paikka}.

By integrating \eqref{yleinen-dh} we obtain
\begin{equation}
\begin{aligned}
&h = \sum _{1\leq i<j\leq k} Q_{ij}g_{ij}
\,\ \mathrm{where}\\ 
&Q_{ij}=a_{1i}'a_{1j}-a_{1j}'a_{1i}+a_{2i}'a_{2j}-a_{2j}'a_{2i}\ .
\end{aligned}
\label{yleinen-h}
\end{equation}
The analysis of the time component of the solutions  will be based on formulas   \eqref{yleinen-det} and \eqref{yleinen-h}.

If there are no spatial constraints then there are $k(k-1)$ conditions for the $2k$ time components of $A$, since every  $p_{ij}$ and $Q_{ij}$ in \eqref{yleinen-det} and \eqref{yleinen-h}  has to be constant. Each spatial constraint of the form \eqref{yleinen-paikka}, however, decreases the number of time constraints by $2$. On the other hand, we need to be able to choose at least two of the spatial variables arbitrarily because of Lemma \ref{koordi} so we expect the number of spatial constraints to be at most $k-2$. In this case there would be $k^2-3k+4$ equations for the $2k$ functions. This means that for $k\leq 4$ we can expect to find solutions but for $k>4$ we will obtain an overdetermined system. We will give a complete analysis of the cases $k=2$, $k=3$, and $k=4$ in this paper. It appears that for $k>4$ there really are no solutions but we could not find a sufficiently neat way to prove this.

In the analysis we often have situations where a certain case reduces to a case of smaller $k$. For future reference we record these simple observations.
\begin{lemma}
Let $\varphi$ be as in \eqref{yrite}. 
\begin{enumerate}
    \item If some $A_j$ is a constant linear combination of other columns then the problem reduces to a similar problem with smaller $k$.
    \item If some $v^i$ is constant the problem reduces.
    \item If some $v^i$ is a constant linear combination of other $v^j$ then the problem reduces.
\end{enumerate}
\label{redu-alempi}
\end{lemma}
\begin{proof} 1. For example let us suppose that $A_k=c_1A_1+\dots+c_{k-1}A_{k-1}$. Then we can set
\[
  \tilde A=\big(A_1,\dots,A_{k-1}\big)\quad\mathrm{and}\quad
  \tilde v^i=v^i+c_iv^k
  \quad,\quad 1\le i<k\ .
\]
Hence $\varphi=Av=\tilde A\tilde v$. 

2. Let us then suppose that $v^k=c=$ constant and let $\tilde v=(v^1,\dots,v^{k-1})$. Then 
\[
  \varphi=Av=\tilde A\tilde v+cA_k
\]
But the conditions in Theorem \ref{kriteeri} do not depend on the term $cA_k$.

3.  If $v^k=c_1v^1+\dots+c_{k-1}v^{k-1}$. Then we can set
\[
  \tilde A=\big(A_1+c_1A_k,\dots,A_{k-1}+c_{k-1}A_k\big)\quad\mathrm{and}\quad
  \tilde v=\big(v^1,\dots, v^{k-1}\big)
  \ .
\]
Hence $\varphi=Av=\tilde A\tilde v$. 
\end{proof}

Also with respect to time one has a simple invariance.
\begin{lemma}
Suppose that some $\varphi=A(t)v(z)$ is a solution and let $\tilde\varphi=M(\theta)\varphi$ or $\tilde\varphi=\hat M(\theta)\varphi$. Then  $\tilde\varphi$ is a solution if and only if $\theta=c_1t+c_0$ where $c_j$ are constants.
\label{aika-inv}
\end{lemma}
\begin{proof}
This is a simple computation using the criteria of Theorem \ref{kriteeri}.
\end{proof}

 Hence, if convenient we can always rotate or reflect our solution with such a matrix. Note that the rotation adds a constant to the vorticity: if $\tilde\varphi=M(\theta_0t)\varphi$, then $\tilde\zeta=2\theta_0+\zeta$.

\subsection{Case $k=2$}
\label{k=2}
Let us briefly recall what happens when $k=2$, see also \cite{CM,maju} for more details. Then according to Lemma \ref{koordi} we can assume without loss of generality that $v(z)=z$. In this case the coordinates $z$ are in fact Lagrangian coordinates, and the corresponding vector field in Eulerian coordinates is given by
\[
  u(x,t)=A'A^{-1}x\, .
\]
The conditions \eqref{yleinen-det} and \eqref{yleinen-h} are now
\[
p_{12}=e\quad ,\quad Q_{12}=c,
\]
where $e$ and $c$ are constants. The solution can be written explicitly for example in the following way. We have $\varphi=Az$ where
\begin{equation}
A=M(\theta)\begin{pmatrix}
            r && r\,a \\
            0 && e/r
    \end{pmatrix} \quad
    ,\quad
    a'= \frac{2e\theta'-c}{r^2}
     \quad
    \mathrm{and}\quad
    \zeta=c/e\ .
\label{2x2-sol}
\end{equation}
Here $\theta $ and $r$ are arbitrary functions of $t$. Note that this is a QR decomposition of the matrix $A$. We may take $e=1$ without loss of generality, so that $A$ is a curve in $\mathbb{SL}(2)$.

Hence one can describe the degree of generality of the solution by saying that one can choose arbitrarily two functions of time. The solution set can also be given in a very different form using complex analysis, like in \cite{CM}. Note that there is no real choice for function $v$, one can say that it is uniquely defined in the sense of Lemma \ref{koordi}. In spite of the relative triviality of this case the well-known Kirchhoff solution is of this form \cite{K}. 

\subsection{Case $k=3$}
Somewhat surprisingly and to the best of our knowledge this case has not been investigated before. Let us first state the main result, which turns out to be a generalization of the above case. To this end we first define the following matrices:
\begin{align*}
    R =\begin{pmatrix}
 r&r\,a_1&r\,a_2\\
 0&1/r&0
  \end{pmatrix}\quad,\quad
  A=M(\theta)R\, .
\end{align*}
Here $r$, $\theta$, $a_1$ and $a_2$ are functions of $t$.

\begin{theorem}
Let $v=\big(z_1,z_2,f(z_2)\big)$ and $\varphi=Av$ where $A$ is as above. Then this gives a solution to Euler equations if
\[
  a_1'=2\theta'/r^2\quad\mathrm{and}\quad
  a_2'=-1/r^2\, .
\]
In this case
\[
   \det(d\varphi)=1
   \quad\mathrm{and}\quad
   \zeta=f'\, .
\]
\label{2x3-th}
\end{theorem}
\begin{proof}
Using the criteria of Theorem \ref{kriteeri} one easily verifies that this is a solution.
\end{proof}

As mentioned, \eqref{2x2-sol} is a special case of this Theorem, obtained by choosing $f(z_2)=c\, z_2$. Note that here, too, we can actually achieve $\det(d\varphi)=1$ so that the coordinates $z$ are in fact real Lagrangian coordinates.

While it is easy to check that we indeed obtain a solution, it is not so easy to prove that this is essentially the most general solution of this form. Note that we have here two arbitrary functions of time, namely $r$ and $\theta$, like in the case $k=2$. However, in addition we have one arbitrary function of one variable in $z$ coordinates, namely $f$. However, there is no canonical form of the solution. For example one could take $\tilde v=\big(z_1,f(z_1),z_2\big)$. Then modifying the matrix $A$ a little we can still get a solution, but the degree of generality remains the same. 

Note that we can find a solution with prescribed vorticity in the  sense that given any $\zeta$ that depends only on $z_2$ we can find the corresponding $f$ by simple integration.

Let us show how to find the complete solution set. For $\fii$ to be a solution, the constraint equations
\begin{equation}
\begin{aligned}
\det (d\fii) &= p_{12}g_{12}+p_{13}g_{13}+p_{23}g_{23},\\
h &= Q_{12}g_{12}+Q_{13}g_{13}+Q_{23}g_{23}
\end{aligned}
\label{2x3}
\end{equation}
have to be independent of time. 
\begin{lemma}
If there are no constraints for the spatial variables, then the problem reduces to the case $k=2$.
\label{redu23}
\end{lemma}
\begin{proof}
Without loss of generality we may assume that $p_{12}\neq 0$. Then we have
\[
  A_3=-\frac{p_{23}}{p_{12}} A_1+
 \frac{p_{13}}{p_{12}} A_2\, .
\]
But if there are no constraints for the spatial variables then each $p_{ij}$ must be constant and the problem reduces by Lemma \ref{redu-alempi}.
\end{proof}

If we have one constraint this can be put in a simpler form.
\begin{lemma}
If there is one constraint for the spatial variables, then without loss of generality we can assume that $g_{23}=0$ and we can choose $p_{12}=1$ and $p_{13}=0$ in \eqref{2x3}.
\label{redu-e}
\end{lemma}
\begin{proof}
By renaming the variables if necessary we may write the constraint as $\alpha _{12}g_{12} + \alpha _{13}g_{13} + g_{23} =0$.
Let 
\[
  H=\begin{pmatrix}
  1&0&0\\\alpha _{13}&1&0\\
  -\alpha _{12}&0&1
  \end{pmatrix}
\]
and put $\tilde{v}=Hv$; then
we compute that $\tilde g_{23}=\alpha _{12}g_{12} + \alpha _{13}g_{13} + g_{23} $. 
Hence we may assume that $g_{23}=0$ in \eqref{2x3} 
so that $p_{12}=e_1$ and $p_{13}=e_2$ where $e_j$ are constants.
By symmetry, we may assume that $e_1\neq 0$ and by scaling we can make it equal to $1$. Then let
\[
H_0=\begin{pmatrix}
    1&0&0\\
    0&1&e_2\\
    0&0&1
\end{pmatrix},\quad \hat{v}=H_0v,\quad A=\hat{A}H_0.
\]
Now $\fii = Av=\hat{A}\hat{v}$, where $\hat{v}$ satisfies $\hat{g}_{23}=0$ and $\hat{A}$ satisfies $\hat{p}_{12}-1=\hat{p}_{13}=0$.
\end{proof}

Hence we expect that there can be only one constraint in the spatial domain. 
\begin{lemma}
If there are two constraints for the spatial variables then either  $\det(d\fii)=0$ or the problem reduces.
\end{lemma}
\begin{proof}
Lemma \ref{redu-alempi} implies that if $\nabla v^j=0$ for some $j$ the problem reduces, so we may suppose that $\nabla v^j\ne 0$.
We have seen that we can assume that one constraint is $g_{23}=0$ and thus $v^3=f(v^2)$ for some $f$. But then the other constraint is of the form
\[
  c_1g_{12}+c_2g_{13}=
  g_{12}\big(c_1+c_2f'(v^2)\big)
  =0\, .
\]
If $g_{12}=0$ then $\det(d\varphi)=0$ by Lemma \ref{minori-riippuvuus}. If $c_1+c_2f'(v^2)=0$ then $v^3=c\,v^2+d$ for some constants $c$ and $d$ and the problem reduces by Lemma \ref{redu-alempi}.
\end{proof}

Hence   there can only be one constraint of the form \eqref{yleinen-paikka}.

\begin{theorem}
The most general solution is of the form given in Theorem \ref{2x3-th}.
\end{theorem}
\begin{proof}
We have seen that without loss of generality we may suppose that the constraint is $
g_{23}=0$. The solution to this equation is $v_3=\tilde f(v_2)$ for an arbitrary function $\tilde f$. By Lemma \ref{koordi}, we may thus assume that $\tilde v=\big(z_1,z_2,\tilde f(z_2)\big)$.
Substituting $g_{23}=0$ to \eqref{2x3} implies also that
\[
h = Q_{12}g_{12}+Q_{13}g_{13}
\]
is independent of time and so  there are constants $c_j$  such that
\[
 Q_{12}=c_1,\quad Q_{13}=c_2.
\]
 But this means that the matrices $(A_1,A_2)$ and $(A_1,A_3)$ must both be solutions to the $2\times 2$ case. Hence by formula \eqref{2x2-sol} we have
 \[
   A=M(\theta)R \quad\mathrm{where}\quad
  R =\begin{pmatrix}
 r&r\,a_1&r\,a_2\\
 0&e_1/r&e_2/r
  \end{pmatrix}
\]
and the functions $a_j$ are given by
\[
  a_j'=\frac{2e_j\theta'-c_j}{r^2}\, .
\]
But according to Lemma \ref{redu-e} we can as well choose $e_1=1$ and $e_2=0$. Moreover, by setting
\[
  v=\big(z_1,z_2,f(z_2)\big)
  =\big(z_1,z_2,c_1z_2+c_2 \tilde f(z_2)\big)
\]
we see that we can also choose $c_1=0$ and $c_2=1$.
\end{proof}

Let us illustrate how a solution of this type might look like. Note that there cannot be any periodic solutions apart from those that can be obtained by considering the case $k=2$.   Since there is a lot of freedom in choosing the various functions, many different kinds of cases are possible. In particular the motion of a single particle can be quite complicated  depending on the choice of the arbitrary functions. But there appears like a wavefront defined by $A_1$: at each $t$, the points whose $z_2$ coordinates are equal are all on the same line parallel to $A_1$. In Figure \ref{23tapaus} we have chosen
\[
  r=1\ ,\ \theta=\sin(t)\ ,\ f= z_2^2/2-z_2^3/3-z_2^4/5 \, .
\]

\begin{figure}
\centering
\includegraphics[height=45mm]{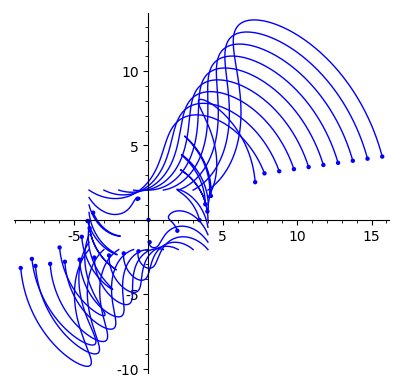}
\caption{The trajectories of some particles. The pointed ends of the trajectories indicate the direction the particles are moving towards.}
\label{23tapaus}
\end{figure}

\section{Case $k=4$, the spatial dependence}

Now we consider solutions of the form \eqref{yrite} with $k=4$.  Let us start by reducing the spatial constraints into a simpler form. The constraints are again as in  \eqref{yleinen-paikka} and  there are now  six terms in the sum. Let us first state the main observations.
\begin{theorem}
If there are less than two constraints or more than two constraints, then either $\mathsf{rank}(dv)<2$ or the problem reduces to the case $k<4$.
\label{vain2}
\end{theorem}
\begin{theorem}
If there are two constraints, then without loss of generality we may assume that they are
\begin{align*}
 &   g_{24}=g_{13}=0&\mathrm{or}&
  &   g_{24}=g_{14}=0&&\mathrm{or}\\
  &  \begin{cases}
     g_{24}+g_{13}=0\\[1mm]
     g_{14}=0
    \end{cases}&\mathrm{or}&
   & \begin{cases}
     g_{24}+g_{13}=0\\[1mm]
     g_{14}-g_{23}=0
    \end{cases}
\end{align*}
\label{ehto-luok}
\end{theorem}
The proof will be based on several Lemmas.

\begin{lemma}
Without loss of generality we may assume that one constraint is of the form 
\[
    g_{24}+c\,g_{13}=0\, .
 \]
\label{redusointi0}
\end{lemma}
\begin{proof}
 Let $\tilde v$ be a vector and let $\tilde g_{ij}$ be the corresponding minors of $d\tilde v$. Then one constraint can be written as
 \[
    \sum \alpha_{ij}\tilde g_{ij}=0\, .
 \]
 Without loss of generality we may assume that $\alpha_{24}=1$. Then let us introduce the following matrix
 \[
   H=\begin{pmatrix}
   1&0&0&0\\
   -\alpha_{14}&1&-\alpha_{34}&0\\
   0&0&1&0\\
   \alpha_{12}&0&-\alpha_{23}&1
   \end{pmatrix}
 \]
 and let $\tilde v=H v$. Then the constraint becomes 
 \[
  g_{24}+\big(\alpha_{13}-\alpha_{12}\alpha_{34}-\alpha_{14}\alpha_{23}\big)  g_{13}=0\, .
 \]
 \end{proof}

\begin{lemma}
If there is only one constraint, the problem reduces to the case $k<4$.
\label{redu1}
\end{lemma}
\begin{proof}
If there is only one spatial constraint, by the above Lemma we may assume it to be $g_{24}+c\,g_{13}=0$. Then the determinant conditions for $A$ imply that the five expressions
\[
p_{12},\quad p_{14},\quad p_{23},\quad p_{34},\quad p_{13}-c\,p_{24}
\]
must all be constant. However, the values of the minors must also satisfy the equation in Lemma \ref{minori-riippuvuus}. But then one column of $A$ must be a constant linear combination of other columns and hence the problem reduces by Lemma \ref{redu-alempi}.
\end{proof}

\begin{lemma}
Without loss of generality we may assume that the two constraints are 
\begin{align*}
 &   g_{13}=g_{24}=0\quad \mathrm{or}\\
  &  \begin{cases}
    g_{24}+c_0g_{13}=0\\
    g_{14}+c_1g_{13}+c_2g_{23}=0
    \end{cases}
\end{align*}
\label{redusointi}
\end{lemma}
\begin{proof}
 By Lemma \ref{redusointi0} we already know that one constraint can be written as $ g_{24}+c_0g_{13}=0$. 
 Hence if the second constraint is $ g_{13}=0$ we have our first case.
 
 Otherwise we can,  without loss of generality, assume that the second constraint is of the form
 \[
  \beta_{12}\tilde g_{12}+
  \beta_{13}\tilde g_{13}+
  \tilde g_{14}+
  \beta_{23}\tilde g_{23}+
  \beta_{34}\tilde g_{34}=0\, .
 \]
 Now let us set
 \begin{align*}
  H=&\begin{pmatrix}
   1&0&0&0\\
   0&1&0&0\\
   0&0&1&0\\
  0& -\beta_{12}&0&1
   \end{pmatrix} \quad\mathrm{if }\quad 
   \beta_{34}=0\quad \mathrm{or}\\
    H=&\begin{pmatrix}
   0&0&1&0\\
   0&1&0&0\\
   1&0&-1/\beta_{34}&0\\
  0& \beta_{23}/\beta_{34}&0&1
   \end{pmatrix} \quad\mathrm{if }\quad 
   \beta_{34}\ne 0\, .
 \end{align*}
 Then we set $\tilde v=Hv$. With this substitution the first constraint is the same as before and the second is of the form $g_{14}+c_1g_{13}+c_2g_{23}=0$ for some constants $c_j$.
\end{proof}

Note that the first case of Theorem \ref{ehto-luok} is the first case of Lemma \ref{redusointi}, and the second case of Theorem \ref{ehto-luok} is obtained from the second case of Lemma \ref{redusointi} by choosing $c_0=c_1=c_2=0$. Then we must analyze how to reduce 
\[
 \begin{cases}
    g_{24}+c_0g_{13}=0\\
    g_{14}+c_1g_{13}+c_2g_{23}=0
    \end{cases}
\]
further when not all constants are zero.

\begin{lemma}
If not all $c_j$ are zero, then without loss of generality we may assume that $c_0\ne 0$ in Lemma \ref{redusointi}.
\label{c0ne0}
\end{lemma}
\begin{proof}
Let us  show that if $c_0=0$ the problem reduces to the previously known cases.

The case $c_0=c_1=0$ and $c_2\ne 0$.
Here we can swap $v^1$ and $v^2$ to obtain $g_{14}=g_{24}+c_2 g_{13}=0$ and the system is in the desired form.

The case $c_0=0$ and $c_1\neq 0$.  Let
\[
H=\begin{pmatrix}
1&-c_2/c_1&0&0\\
0&1&0&0\\
0&0&1&-1/c_1\\
0&0&0&1
\end{pmatrix}
\]
and let $\tilde{v}=Hv$. After this transformation we have $g_{24}=g_{13}=0$, the first case in Theorem  \ref{ehto-luok}.
\end{proof}

We are  finally ready for the proof of Theorem \ref{ehto-luok}.

\begin{proof} (Theorem \ref{ehto-luok}). 
 The first case of the classification in Theorem \ref{ehto-luok} is the first case of Lemma \ref{redusointi}. The second and third case of the Theorem are obtained from  the second case of Lemma  \ref{redusointi} by choosing  $c_0=c_1=c_2=0$, or $c_0=1$ and $c_1=c_2=0$. 
 
 Let $\tilde v$ be our vector and let us denote the corresponding minors by $\tilde g_{ij}$. We have to show that in the remaining cases we obtain the fourth case or another known case.  By Lemma \ref{c0ne0} we may assume that $c_0=1$ and hence we have to reduce constraints of the form 
 \[
  \begin{cases}
   \tilde  g_{24}+\tilde g_{13}=0\\
    \tilde g_{14}+c_1\tilde g_{13}+c_2\tilde g_{23}=0
    \end{cases}
\]
to a simpler form. Let 
\[
H=\begin{pmatrix}
\beta_2&-\beta_1&0&0\\
1&-1&0&0\\
0&0&-1&-1\\
0&0&\beta_1&\beta_2
\end{pmatrix}
\]
and $\tilde v=Hv$. Then 
\begin{align*}
    \tilde g_{24}+\tilde g_{13}&=
    (\beta_1-\beta_2)(g_{24}+g_{13})=0\, ,\\
     \tilde g_{14}+c_1\tilde g_{13}+c_2\tilde g_{23}&=
 (\beta_2^2-c_1\beta_2-c_2)g_{14}+(2\beta_1\beta_2-c_1(\beta_1+\beta_2)-2c_2)g_{13}\\
 & +(\beta_1^2-c_1\beta_1-c_2)g_{23}=0\, .
\end{align*}
If the polynomial $p=x^2-c_1x-c_2$ has distinct real roots, then choosing  $\beta_j$ to be these roots we obtain  $g_{24}=g_{13}=0$, the first case in Theorem \ref{ehto-luok}.
If there is a double root then choosing  $\beta_1=c_1/2$ we get $g_{14}=0$, which leads to  the third case. 
If the roots are complex we choose  $\beta_2=(2c_2+c_1\beta_1)/(2\beta_1-c_1)$ which leads to
\[
(4c_2+c_1^2)g_{14}+(c_1-2\beta_1)^2\,g_{23}=0\, .
\]
Since $4c_2+c_1^2<0$, we can further reduce this to $g_{14}- g_{23}=0$ by scaling.
\end{proof}

Now in Theorem \ref{ehto-luok} we have four PDE systems for the vector $v$. So the next task is to find the general solutions to these systems. However, one of the cases can be discarded.

\begin{lemma}
 If $g_{14}=g_{24}=0$ then the problem  reduces to the case $k<4$.
\end{lemma}
\begin{proof} Renaming the variables we can write the system as  $g_{34}=g_{24}=0$. If $\nabla v^4=0$ the problem reduces by Lemma \ref{redu-alempi}. If $\nabla v^4\ne 0$ then Lemma \ref{minori-riippuvuus}   implies that $g_{23}=0$ so the conditions for $A$ are
 \begin{align*}
     p_{12}&=e_1, \quad p_{13}=e_2, \quad p_{14}=e_3 \\
     Q_{12}&=c_1, \quad Q_{13}=c_2, \quad Q_{14}=c_3
 \end{align*}
 for some constants $e_j$, $c_j$. This means that $(A_1,A_2)$, $(A_1,A_3)$, and $(A_1,A_4)$ are all solutions to the $2\times 2$ case. Hence according to formula \eqref{2x2-sol} we can write $A=M(\theta)R$, where
 \begin{align*}
    R =\begin{pmatrix}
 r&r\,a_1&r\,a_2&r\,a_3\\
 0&e_1/r&e_2/r&e_3/r
  \end{pmatrix},
\end{align*}
and $a_j'=(2e_j\theta'-c_j)/r^2$. But then we can write  $a_j=e_jg_1(t)+c_jg_2(t)+d_j$ where $g_j$ are some functions and $d_j$ are constants. Replacing $v^1$ by $v^1+d_1v^2+d_2v^3+d_3v^4$ we may assume that $d_j=0$. This implies that a constant linear combination of  $A_2$, $A_3$, and $A_4$ is zero, and thus  the problem reduces by Lemma \ref{redu-alempi}.
\end{proof}

Let us then find the solutions in the remaining cases.
\begin{theorem}
In the relevant cases  of Theorem \ref{ehto-luok} we have the following solutions where the functions $f_j$ are arbitrary.
\begin{enumerate}
    \item If $g_{13}=g_{24}=0$ then we can take $$v=\big(z_1,z_2,f_1(z_1),f_2(z_2)\big).$$
    \item If $g_{14}=g_{24}+g_{13}=0$ then we can take $$v=\big(z_1,z_2,z_2f_1'(z_1)+f_2(z_1),f_1(z_1)\big).$$
    \item If $g_{14}-g_{23}=g_{24}+g_{13}=0$ then we can take $v=\big(z_1,z_2,v^3,v^4\big)$ where $v^3$ and $v^4$ satisfy the anti CR system.
\end{enumerate}
\label{luokittelu}
\end{theorem}

\begin{proof}
In each of the three cases we must have $g_{12}\neq 0$. Indeed, otherwise the equalities of Lemma \ref{minori-riippuvuus}, combined with the conditions of any of the three cases, imply that all the minors are zero and thus $\det(d\fii)$ is zero. Therefore by Lemma \ref{koordi} we may choose a labelling with $v^1=z_1$ and $v^2=z_2$.

Now we prove each case of the Theorem:

 Case 1.  The general solution to $g_{13}=g_{24}=0$ can be written as $v^3=f_1(z_1)$ and $v^4=f_2(z_2)$.

 Case 2. The equation $g_{14}=0$ implies that $v^4=f_1(z_1)$ where $f_1$ is arbitrary. Then the second equation is $v^3_{01}=f'(z_1)$. Then we integrate to get the result.
 
 Case 3. Simply substituting $v=\big(z_1,z_2,v^3,v^4\big)$  we obtain the anti CR system.
\end{proof}

It is difficult to show directly that these three cases are actually different, i.e. they cannot be reduced to each other. We will show this later in Lemma \ref{vika}.

Then we should prove Theorem \ref{vain2}. We now already know that two constraints can be reduced to the cases in Theorem \ref{luokittelu}. Hence we should show that if we add further equations the problem reduces to the case $k<4$.

\begin{proof} (Theorem \ref{vain2}). If there is only one constraint the problem reduces by Lemma \ref{redu1}. If there are three constraints we have the following cases.

Case 1 of Theorem \ref{luokittelu}. Without loss of generality we may suppose that the three constraints are
\begin{align*}
  &  g_{24}=g_{13}=0\, ,\\
  &  c_1g_{12}+c_2g_{14}+c_3g_{23}+c_4g_{34}=0\, .
\end{align*}
Since we know that $v=\big(z_1,z_2,f_1(z_1),f_2(z_2)\big)$, then simply substituting this to the third equation gives
\[
c_1+c_2f_2'-c_3f_1'+c_4f_1'f_2'=0\, .
\]
It is straightforward to check the solutions are affine and hence problem reduces.

Case 2 of Theorem \ref{luokittelu}. Consider the equations
\begin{align*}
   & g_{24}-g_{13}=g_{14}=0\, ,\\
  &  c_1g_{12}+c_2g_{13}+c_3g_{23}+c_4g_{34}=0\, .
\end{align*}
We know that $v=\big(z_1,z_2,z_2f_1'(z_1)+f_2(z_1),f_1(z_1)\big)$. Hence the third equation is
\[
  c_1+c_2f_1'-c_3\big(f_2'+z_2(f_1')^2\big)-c_4f_1''=0\, .
\]
Again it is easy to check that the solutions are affine and the problem reduces.

Case 3 of Theorem \ref{luokittelu}. Now $v=(z_1,z_2,v^3,v^4)$ where $(v^3,v^4)$ is an anti CR map. The first two constraints are thus the anti CR system and the third constraint can be written as
\[
    c_1g_{12}+c_2g_{13}+c_3g_{23}+c_4g_{34}=0\, .
\]
Using the anti CR system to eliminate $v^4$ we thus obtain a system
\begin{align*}
    &    \Delta v^3=0\, ,\\
     &   c_1+c_2v^3_{01}-c_3v^3_{10}-c_4|\nabla v^3|^2=0\, .
\end{align*}
Using \textsf{rifsimp} one easily verifies that the solutions are necessarily affine and thus the problem reduces.
\end{proof}

\subsection{Comparison of cases 1 and 3}
Let us point out a relationship between cases 1 and 3, which is in a way hidden in the formulation given. In case 3 we have thus $v=\big(z_1,z_2,v^3,v^4\big)$ where $(v^3,v^4)$ is an anti CR map and in case 1 $v=\big(z_1,z_2,f_1(z_1),f_2(z_2)\big)$ where $f_j$ are arbitrary. But now recall that the general solution of the one dimensional wave equation $u_{11}=0$ can be written as
\[
   u(z_1,z_2)=f_1(z_1)+f_2(z_2)\, .
\]
So in a way case 3 is an elliptic case and case 1 is a hyperbolic case.
In fact we could have used a different basic form in Theorem \ref{luokittelu} to make the connection more explicit. Like in case 3 we have an anti CR system, in case 1 we could have used the coupled wave system
\begin{align*}
    &  v^3_{10}+v^4_{01}=0\, ,\\
    & v^3_{01}+v^4_{10}=0\, .
\end{align*}
In this way $v^3$ and $v^4$ are both solutions to the wave equation $u_{20}-u_{02}=0$. However, a simple change of variables leads to the form given in Theorem \ref{luokittelu}, which is more convenient to represent the solutions to Euler equations. Taking this point of view we thus obtain a new family of solutions, case 1, from the old one, case 3, by changing one sign in the anti CR system. 
We will see that this elliptic/hyperbolic character also shows up when we compute the corresponding vorticities below.

\section{$k=4$, the time dependence}
Now we begin the analysis of the time component $A$ in the three relevant cases shown in Theorem \ref{luokittelu}.
\subsection{Case 3}
In this case the spatial constraints are
\begin{equation}
    g_{14}-g_{23}=g_{24}+g_{13}=0
    \label{ell-ehdot}
\end{equation}
and we have seen that we may take $v=(z_1,z_2,f^1,f^2)$ for some anti CR map $f=(f^1,f^2)$. This has already been studied previously \cite{AY,AC,CM,maju}. A famous example of this case is the Gerstner map \cite{G}:
\[
  \varphi_G=z+M(\mu t) \,f_G\quad \mathrm{where}\quad
 f_G=\frac{e^{kz_2}}{k}\begin{pmatrix}
  \sin(k z_1 )\\
  -\cos(k z_1) 
  \end{pmatrix}\, .
\]
In this case we compute
\[
  \det(d\varphi_G)=1-e^{2kz_2}\quad\mathrm{and}\quad
  \zeta_G=\frac{2\mu e^{2kz_2 }}{1-e^{2kz_2}}.
\]
In general  we have the following result.
\begin{theorem}
Let $f$ be any anti CR map such that $1-|\nabla f^1|^2\neq 0$ in $D$. If
\begin{equation}
\varphi=z+M(\mu t)f,
\label{ell-ratk1}
\end{equation}
then $\fii$ gives a solution to Euler equations and in this case
\[
 \det(d\varphi)=1-|\nabla f^1|^2\quad\mathrm{and}\quad
  \zeta=\frac{2\mu |\nabla f^1|^2}{1-|\nabla f^1|^2}\, .
\]
\label{ell-ratk}
\end{theorem}
\begin{proof}
This is again a simple computation using the criteria of Theorem \ref{kriteeri}.
\end{proof}

Note that now we have practically no choices in the time domain, but in some sense more choices in the spatial domain than in the previous cases. 

In Figure \ref{ellitapaus} we have an example of this case with 
\[
   f=\big( z_1^2-z_2^2+1/20, -2z_1z_2\big)\ ,\ \mu =1\ ,\ \theta_0=1/2\, ,
\]
where $\theta_0$ is the coefficient of the implicit rotation matrix $M(\theta_0 t)$ that we can premultiply the solution by, according to Lemma \ref{aika-inv}.

Before the proof that this is indeed the most general form of the solution let us make a few comments of the form of the solution. Previously solutions of this type have been given in different forms and so let us indicate what is the relationship between various formulations. Let $w=(w^1,w^2)$ and $\hat w=(\hat w^1,\hat w^2)$ and let $\theta_0$ and $\mu_0$ be some constants. Then one could look for the solutions of the form
\begin{equation}
   \varphi=M(\theta_0t)w+M(\mu_0 t)\hat w\, .
   \label{tavallaan}
\end{equation}

As explained in \cite{maju} in the  PDE system for $w$ and $\hat w$, namely the system \eqref{ell-ehdot},  one can for example give $w$ arbitrarily and then solve the corresponding $\hat w$. This is a regular elliptic system for $\hat w$. Note that here (anti) CR maps play no role a priori. However, it has been known that if $w$ is a CR map and $\hat w$ an anti CR map then this provides a solution to the equations \eqref{ell-ehdot} \cite{AO2,AC}. In fact it seems that all the solutions that were known before \cite{maju} assumed the harmonicity of $w$ and $\hat w$.

\begin{figure}
\centering
\includegraphics[height=45mm]{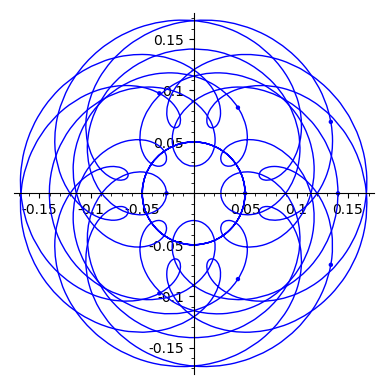}
\caption{Some trajectories in an example solution of case 3.}
\label{ellitapaus}
\end{figure}

Anyway we have the following simple observation.
\begin{lemma}
If there is a solution of the form \eqref{tavallaan} then there is 
 an anti CR map $f$ such that $\hat w=f\circ w$. 
 \label{ell-redu}
\end{lemma}
Hence even if $w$ and $\hat w$ are not (anti) CR maps they are connected by an anti CR map.

\begin{proof}
Without loss of generality we may suppose that $\det(dw)\ne0$. Hence there is some map $f$ such that $\hat w=f\circ w$. Then substituting this to the system shows that $f$ must be an anti CR map.
\end{proof}

Now using $w$ as new coordinates we obtain solutions which are as given in Theorem \ref{ell-ratk}. Note that the form \eqref{tavallaan} can be very useful because it may be possible or more convenient to compute $w$ and $\hat w$ directly, in which case typically $f$ is not explicitly known.

Let us then turn to the proof that the most general solution is given by Theorem \ref{ell-ratk}, taking into account Lemmas \ref{koordi} and \ref{aika-inv} as always. 
Using \eqref{ell-ehdot}, the conditions of Theorem \ref{kriteeri} give that
\begin{align*}
    \det (d\fii) &= p_{12}g_{12} + p_{34}g_{34} + (p_{13}-p_{24})g_{13} + (p_{14}+p_{23})g_{14},
  \\
    h &= Q_{12}g_{12} + Q_{34}g_{34} + (Q_{13}-Q_{24})g_{13} + (Q_{14}+Q_{23})g_{14}
\end{align*}
are constant w.r.t. time. Hence there are constants $e_j$ and $c_j$ such that
\begin{equation}
\begin{aligned}
    p_{12}&=e_1&
    p_{34}&=e_2&
    p_{13}-p_{24}&=e_3&
    p_{14}+p_{23}&=e_4\\
    Q_{12}&=c_1&
    Q_{34}&=c_2&
    Q_{13}-Q_{24}&=c_3&
    Q_{14}+Q_{23}&=c_4\ .
\end{aligned}
\label{ell-ec}
\end{equation}

First we can reduce the problem to a simpler form.

\begin{lemma}
Without loss of generality we may assume that $e_1\neq 0$.
\end{lemma}
\begin{proof}
Suppose that $e_1=0$. Due to symmetry we only need to consider the case where also $e_2=0$. Let $v=H\tilde{v}$  where
\[
H=\begin{pmatrix}
1&0&0&1\\
0&1&1&0\\
1&0&1&0\\
0&-1&0&1 \end{pmatrix}.
\]
Now $\tilde{v}$ satisfies the equations \eqref{ell-ehdot} if and only if $v$ satisfies them. 
Then let $\tilde{A}=AH$; thus we can write $\varphi=Av=\tilde A\tilde v$. Then we obtain
\begin{align*}
\tilde{p}_{12} &= p_{12}-p_{23}-p_{14}-p_{34} = -p_{23}-p_{14} = -e_4, \\
\tilde{p}_{34} &= -p_{12}-p_{13}+p_{24}+p_{34} = p_{24}-p_{13} = -e_3.
\end{align*}
Here $e_3$ and $e_4$ cannot both be zero because otherwise $\det(d\fii)=0$. Thus, after this transformation we have $\tilde{e}_1=\tilde{p}_{12}\neq 0$ or $\tilde{e}_2=\tilde{p}_{34}\neq 0$, and by symmetry we may assume the former.
\end{proof}

\begin{lemma}
Without loss of generality we can choose $e_1=e_2=1$ and $e_3= e_4=0$ in \eqref{ell-ec}.
\end{lemma}
\begin{proof}
By the previous Lemma we can assume that $e_1\ne 0$ and by scaling we can assume that $p_{12}=e_1=1$; hence $B=\big(A_1,A_2\big)\in \mathbb{SL}(2)$. 
Let
\begin{align*}
    & A=\tilde A H
     \quad \mathrm{and}\quad
     \tilde v=Hv
     \quad\mathrm{where}\\
    & \tilde A=\big( B\ , BM(\beta)\big)
     \quad \mathrm{and}\quad
       H=\begin{pmatrix}
  1&0&-e_4/2&e_3/2\\
  0&1&e_3/2&e_4/2\\
  0&0&1&0\\
  0&0&0&1
  \end{pmatrix}\ .
\end{align*}
Here $\beta$ is some function. 
Then for $A$ we have  $ p_{13}-p_{24}=e_3$ and $
    p_{14}+p_{23}=e_4 $.
    For $\tilde A$ we have  $\tilde p_{13}- \tilde p_{24}=\tilde p_{23}+ \tilde p_{14}=0 $ and $\tilde p_{34}=1$. Since $\tilde{v}$ satisfies the equations \eqref{ell-ehdot} if and only if $v$ satisfies them,  we can write $\varphi=Av=\tilde A\tilde v$.
\end{proof}

\begin{theorem}
When the spatial constraints are  \eqref{ell-ehdot}, then the most general solution is given by \eqref{ell-ratk1} .
\end{theorem}
\begin{proof}
 By the previous Lemma we may suppose that 
\[
   A=\big( B\ ,\ BM(\beta)\big)
\]
where $B\in \mathbb{SL}(2)$. If $B=M(\mu)$, we obtain
immediately that $\mu'$ and $\beta'$ are constants and we get the required form using Lemma \ref{aika-inv}.

If $B$ is not a rotation, it can be written as
\[
   B = \cosh(s)M(\mu) + \sinh(s)\hat M(\theta)
\]
where $s \neq 0$, $\mu$, and $\theta$ are some functions.
The conditions in the second row of \eqref{ell-ec} give the following equations:
\begin{align*}
   &  (\mu'-\theta')\cosh(2s)+\mu'+\theta'=c_1&
    &  2\beta'\cosh(2s)=c_2-c_1\\
    & 2\beta' \sin(\theta-\mu-\beta)\sinh(2s)=-c_3&
     & 2\beta' \cos(\theta-\mu-\beta)\sinh(2s)=c_4
\end{align*}
Evidently $\theta-\mu-\beta$ must be constant. It follows that clearly $\beta '$ and $s$ are constants, and further that $\mu '$ and $\theta '$ are constants. Hence we can write
\[
\mu = \mu _1t+\mu _0 \quad , \quad \theta = \theta _1t + \theta _0 \quad \mathrm{and} \quad \beta =(\theta _1-\mu _1)t+\beta _0,
\]
where $\mu _0$, $\mu _1$, $\theta _0$, $\theta _1$, and $\beta _0$ are constants. Let us set $\beta _1=\theta _1-\mu _1$. Using Lemma \ref{aika-inv} we can premultiply by the matrix $M(-\mu_1t-\mu_0)$ so that without loss of generality we may assume that
\begin{align*}
      B=&\cosh(s)I+\sinh(s) \hat M(\beta_1t+\theta_0-\mu_0)\, ,\\
       BM(\beta)=&\cosh(s)M(\beta_1t+\beta_0)+\sinh(s)\hat M(\theta_0-\mu_0-\beta _0)\, .
\end{align*}
Hence 
\begin{align*}
  \varphi&=Av=Bz+BM(\beta)f\\
  &=\cosh(s)z+\sinh(s)\hat M(\theta_0-\mu_0-\beta _0)f\\
  &+M(\beta_1t)\Big(\sinh(s) \hat M(\theta_0-\mu_0) z+\cosh(s) M(\beta_0) f\Big)\\
  &=w+ M(\beta_1t)\hat w\, .
\end{align*}
Now it is straightforward to check that $(w,\hat w)$ satisfies the system \eqref{ell-ehdot}, and hence by Lemma \ref{ell-redu} we may take $w$ as new coordinates, which then gives the required form.
\end{proof}

\subsection{Case 1}
Here the spatial constraints are 
$g_{13}=g_{24}=0$
and the equations for the time component are
\begin{align*}
    \det(d\fii)&=p_{12}g_{12}+p_{14}g_{14}+p_{23}g_{23}+p_{34}g_{34}=\textrm{ constant w.r.t. }t\\
    h&=Q_{12}g_{12}+Q_{14}g_{14}+Q_{23}g_{23}+Q_{34}g_{34}=\textrm{ constant w.r.t. }t.
\end{align*}
Thus we have constants $e_j$, $c_j$ such that
\begin{equation}
\begin{aligned}
    p_{12}=&e_1\ ,\  p_{34}=e_2\ ,\   p_{23}=e_3\ ,\  p_{14}=e_4\ ,\\
    Q_{12}=&c_1\ ,\  Q_{34}=c_2\ ,\  Q_{23}=c_3\ ,\  Q_{14}=c_4\ .
\end{aligned} 
\label{hyp-ehto}
\end{equation}
Note that by Lemma  \ref{minori-riippuvuus} we have $p_{13}p_{24}=e_1e_2+e_3e_4=$ constant.
\begin{lemma}
We may choose $e_1=1$ and $e_3=e_4=0$ without loss of generality.
\end{lemma}
\begin{proof}
Due to symmetry, we may assume that $p_{12}\neq 0$ and further that $p_{12}=e_1=1$.
 Let $\ell$ be some function and let 
\begin{align*}
    & A=\tilde A H\quad\mathrm{where}\\
     & \tilde A=\Big( A_1, A_2,\ell A_2, A_1/\ell\Big)
     \quad \mathrm{and}\quad
       H=\begin{pmatrix}
  1&0&-e_3&0\\
  0&1&0&e_4\\
  0&0&1&0\\
  0&0&0&1
  \end{pmatrix}\, .
\end{align*}
Then for $A$ we have  $ p_{23}=e_3$ and $p_{14}=e_4 $ while for $\tilde A$ we have  $\tilde p_{23}= \tilde p_{14}=0 $. Now if $\tilde v=Hv$ then we still have $\tilde g_{13}=\tilde g_{24}=0$. Hence we can write $\varphi=Av=\tilde A\tilde v$.
\end{proof}

\begin{theorem}
If $v$ satisfies $g_{13}=g_{24}=0$, then the most general solution $\varphi=Av$ is given by 
\[
A=\begin{pmatrix}
    e^{ct} & 0 & 0 &  e^{-ct} \\
    0 &  e^{-ct} &  e^{ct} & 0
\end{pmatrix},
\]
where $c$ is a constant, and $v=\big(z_1,z_2,f_1(z_1),f_2(z_2)\big)$, where $1-f_1'f_2'\neq 0$ in $D$. In this case 
\[
  \det(d\varphi)=1-f_1'f_2'\quad \mathrm{and}\quad
  \zeta=-\frac{2c\big(f_1'+
  f_2'\big)}{1-f_1'f_2'}\, .
\]
\label{hyper-aika}
\end{theorem}
\begin{proof}
By the previous Lemma we may assume that $ \det(A_1,A_2)=1 $,  $A_3=\ell A_2$ and $A_4=A_1/\ell$. The conditions in the second row of \eqref{hyp-ehto} give the following conditions for $A$:
\begin{align*}
    \ell'|A_2|^2 &=c_3&
    \ell'|A_1|^2 &=-c_4\ell^2\\
   2 \ell'\langle A_1,A_2\rangle+(c_1+c_2)\ell &=0&
   \langle A_1',A_2\rangle-\langle A_1,A_2'\rangle&=-c_1\ .
\end{align*}
Hence $(\ell'/\ell)^2=$ constant and thus we can take $\ell(t)=e^{2ct}$. Then we see that $\langle A_1,A_2\rangle$ is constant and hence we may write
\[
  A_1=e^{ct}M(\theta)\hat A_1\quad,\quad
   A_2=e^{-ct}M(\theta)\hat A_2
\]
where $\hat A_j$ are constant vectors. Now we check that in fact $\theta$ is a linear function so by Lemma \ref{aika-inv} we can drop $M$. Then by a constant rotation we can assume that $\hat A_1=(r,0)$ so that at present we can write
\[
A=\begin{pmatrix}
   r & b \\ 0 & 1/r
\end{pmatrix}\begin{pmatrix}
    e^{ct} & 0 & 0 &  e^{-ct} \\
    0 &  e^{-ct} &  e^{ct} & 0
\end{pmatrix},
\]
where $r$ and $b$ are constants. We can still assume $r=1$ and $b=0$ by introducing
\[
      H=\begin{pmatrix}
  1/r&0&-b&0\\
  0&r&0&0\\
  0&0&r&0\\
  0&-b&0&1/r
  \end{pmatrix}
\]
and setting $\tilde A=AH$ and $v=H\tilde v$. 
\end{proof}

 Now a single particle has a fairly straightforward trajectory: when the absolute value of $t$ is large, then it is approaching the origin approximately along the line parallel to the vector $(f_2(z_2),z_2)$ if $t$ is negative, or moving away from the origin approximately along the line parallel to the vector $(z_1,f_1(z_1))$ if $t$ is positive. Thus its trajectory resembles a hyperbola. When $t$ is negative, the particles are grouped according to the $z_2$ coordinate, and when $t$ is positive, they are grouped according to the $z_1$ coordinate. Then, if we rotate the solution with $M(\theta_0 t)$ the particles are also rotating around the origin as they move towards it or away from it. In Figure \ref{hypertapaus} we have chosen $c=1$, $\theta_0=1/2$, $f_1= 3\cos(3z_1)/(2+2z_1^2)$ and 
$f_2 = -\sin(3z_2/2)/4+\sin(4z_2)/2$ for an example.

\begin{figure}
\centering
\includegraphics[height=45mm]{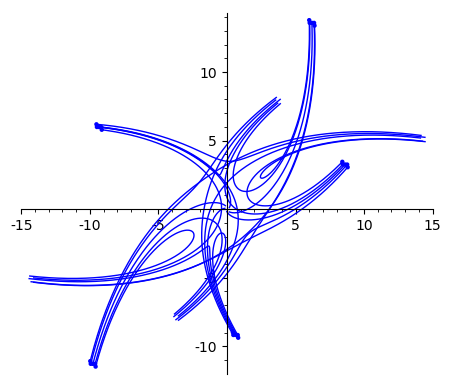}
\caption{Trajectories of an example solution of case 1.}
\label{hypertapaus}
\end{figure}

\subsection{Case 2}

Now we have the following equations for the time component:
\begin{align*}
  &  p_{12}=e_1,\, p_{34}=e_2,\, p_{13}-p_{24}=e_3,\, p_{23}=e_4,\\
   & Q_{12}=c_1,\, Q_{34}=c_2,\, Q_{13}-Q_{24}=c_3,\, Q_{23}=c_4.
\end{align*}

\begin{lemma}
Without loss of generality we may take $e_4=1$ and $e_1=e_2=e_3=0$.
\end{lemma}
\begin{proof}
If $e_4=0$, Lemma \ref{minori-riippuvuus} implies  that $p_{13}$ and $p_{24}$ are constants. Thus the problem reduces.

Hence we may assume that $e_4=1$. Let $\ell$ be some function and let 
\begin{align*}
    & A=\tilde A H\quad\mathrm{where}\\
     & \tilde A=\Big( \ell A_2, A_2,A_3,\ell A_3\Big)
     \quad \mathrm{and}\quad
       H=\begin{pmatrix}
  1&0&0&0\\
  e_3/2&1&0&-e_2\\
  -e_1&0&1&-e_3/2\\
  0&0&0&1
  \end{pmatrix}\, .
\end{align*}
Then for $A$ we have $p_{12}=e_1$, $ p_{34}=e_2$ and $ p_{13}-p_{24}=e_3 $ while for $\tilde A$ we have  $\tilde p_{12}= \tilde p_{34}=\tilde p_{13}-\tilde p_{24}=0 $. Now if $\tilde v=Hv$ then we still have $\tilde g_{13}+\tilde g_{24}=\tilde g_{14}=0$. Hence we can write $\varphi=Av=\tilde A\tilde v$.
\end{proof}

\begin{theorem}
If $v$ satisfies $ g_{13}+g_{24}=
    g_{14}=0$, 
then  the solution $\varphi=Av$  is given by 
\[
A=\begin{pmatrix}
    t & 1 & 0 & 0 \\
    0 & 0 & 1 & t
\end{pmatrix},
\]
where $v=\big(z_1,z_2,z_2f_1'(z_1)+f_2(z_1),f_1(z_1)\big)$ with $-z_2f_1''-f_2'\neq 0$. Moreover 
\[
  \det(d\varphi)=-z_2f_1''-f_2'\quad \mathrm{and}\quad
  \zeta=\frac{(f_1')^2}{z_2f_1''+f_2'}\, .
\]
\label{para-ratk}
\end{theorem}
\begin{proof}
By the previous Lemma we may assume that    $\det(A_2,A_3)=1$,  $A_1=\ell A_2$ and $A_4=\ell A_3$. Hence the conditions for $A$ can be written as
\begin{align*}
    \ell'|A_2|^2 &=c_1&
    \ell'|A_3|^2 &=-c_2\\
    \ell'\langle A_2,A_3\rangle &=c_3/2&
   \langle A_2',A_3\rangle-\langle A_2,A_3'\rangle&=c_4.
\end{align*}
Evidently $|A_2|$, $|A_3|$ and $\ell'$ are constants. Then it is easy to compute that the solution is of the form
\[
A=\begin{pmatrix}
    b_1t+b_0 & 1 & 0 & 0 \\
    0 & 0 & 1 & b_1t+b_0
\end{pmatrix},
\]
where we may assume $b_1\neq 0$. Now the transformation $\tilde{A}=AH$, $v=H\tilde{v}$, where
\[
H=\begin{pmatrix}
1/b_1&0&0&0\\
-b_0/b_1&1&0&0\\
0&0&1&-b_0/b_1\\
0&0&0&1/b_1
\end{pmatrix},
\]
preserves the spatial constraints and gives the desired form to the time component.
\end{proof}

Except for the possible rotation of constant speed, the trajectory of each particle is a line segment parallel to the vector $(z_1,f_1(z_1))$. Figure \ref{paratapaus} gives an example of this case, with $f_1=\cos(z_1)$, $f_2=z_1^2-20z_1$, and $\theta_0=-1/40$.

\begin{figure}
\centering
\includegraphics[height=45mm]{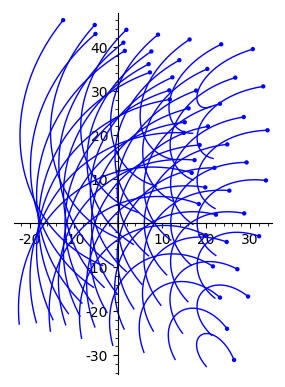}
\caption{One possible solution of case 2.}
\label{paratapaus}
\end{figure}

Now that Theorems \ref{ell-ratk}, \ref{hyper-aika}, and \ref{para-ratk} give the time component for each of the cases obtained in Theorem \ref{luokittelu}, it is easy to show that these three cases are inequivalent.

\begin{lemma}
The three cases of Theorem \ref{luokittelu} cannot be reduced to each other by a linear transformation $\tilde{v}=Hv$.
\label{vika}
\end{lemma}
\begin{proof}
Let us prove that cases 1 and 2 are inequivalent. The proof for the rest of the pairs is similar. If cases 1 and 2 were equivalent, then there would be a solution $\fii = Av = \tilde{A}\tilde{v}$, where $v$ is an instance of case 1 and $\tilde{v}=Hv$ an instance of case 2. But then we would also have $\tilde{A}=AH^{-1}$, where $A$ is a solution to case 1 given by Theorem \ref{hyper-aika} and $\tilde{A}$ a solution to case 2 given by Theorem \ref{para-ratk}, and there is clearly no matrix $H$ that can satisfy this.
\end{proof}

\section{Vorticity}
Let us finally say a few words about vorticity. Above we have computed some families of solutions and the corresponding vorticities. However, one could also ask if one can find a solution with a prescribed vorticity. Let us examine each of the relevant cases. 

First let us consider the situation in Theorem \ref{hyper-aika}. Our solution is $\varphi=Av$ where 
\[
  A=\begin{pmatrix}
    e^{ct} & 0 & 0 &  e^{-ct} \\
    0 &  e^{-ct} &  e^{ct} & 0
\end{pmatrix}\quad\mathrm{and}\quad
v=\big(z_1,z_2,f_1(z_1),f_2(z_2)\big)
\]
and the vorticity is given by
\begin{equation}
 \zeta=-\frac{2c\big(f_1'+
  f_2'\big)}{1-f_1'f_2'}\, .
\label{vorti1}    
\end{equation}
\begin{lemma}
If the vorticity is given by \eqref{vorti1}, then it is a solution to the following PDE:
\[
 (\zeta^2+4c^2)\zeta_{11}-2\zeta\zeta_{10}\zeta_{01}=0\, .
\]
\end{lemma}
\begin{proof}
Note that the equation \eqref{vorti1} is not ''overdetermined'' in the usual sense. However, the right hand side is of the separation of variables type, so the left hand side cannot be completely arbitrary. Giving this equation to \textsf{rifsimp} and specifying the elimination order that eliminates the functions $f_j$ produces the given PDE.
\end{proof}

Note that we can actually find one family of solutions to the vorticity equation:
\[
 \zeta=2c\tan(d_0+d_1z_1+d_2z_2)\, .
\]
Here $d_j$ are constants. Of course this is not the general solution. Note also that the equation for vorticity is a kind of a nonlinear wave equation.

Let us then consider Theorem \ref{ell-ratk}. Now we have 
\[
\varphi=z+M(\mu t)f\, ,
\]
where $f$ is an anti CR map and
\begin{equation}
  \zeta=\frac{2\mu |\nabla f^1|^2}{1-|\nabla f^1|^2}.
\label{vorti-ell}    
\end{equation}
\begin{lemma}
If the vorticity is given by \eqref{vorti-ell}, then
\[
  \zeta(2\mu+\zeta)\Delta \zeta+2(\mu+\zeta)|\nabla\zeta|^2=0\, .
\]
\end{lemma}
\begin{proof}
Since $f$ is an anti CR map we have also $\Delta f^1=0$, so again $\zeta$ cannot be arbitrary. Using \textsf{rifsimp} to eliminate $f^1$ we obtain the above PDE for $\zeta$.
\end{proof}

Again one can find a specific family of solutions:
\[
 \zeta=-\mu\big(1+\tanh(d_0+d_1z_1+d_2z_2)\big)\, .
\]
In this case the vorticity equation is a nonlinear elliptic equation. 

Finally we have the case of Theorem \ref{para-ratk}. Now $\varphi=Av$ where 
\begin{align*}
A&=\begin{pmatrix}
    t & 1 & 0 & 0 \\
    0 & 0 & 1 & t
\end{pmatrix}\quad\mathrm{and}\\
 v&=\big(z_1,z_2,z_2f_1'(z_1)+f_2(z_1),f_1(z_1)\big)\, ,
\end{align*}
and 
\begin{equation}
     \zeta=\frac{(f_1')^2}{z_2f_1''+f_2'}\, .
     \label{vorti-para}
\end{equation}
 \begin{lemma}
 If the vorticity is given by \eqref{vorti-para}, then 
 \[
   \zeta=\frac{1}{g_1(z_1)+z_2 g_2(z_1)}\, ,
 \]
 where the functions $g_j$ are arbitrary.
 \end{lemma}
 \begin{proof}
 Eliminating $f_1$ and $f_2$ with \textsf{rifsimp} we obtain $ \zeta\zeta_{02}-2\zeta_{01}^2=0$, whose general solution is given above.
 \end{proof}


\end{document}